\documentclass[11 pt]{amsart}
\usepackage{amsmath,amsthm,amssymb,graphicx,mathrsfs}
\usepackage{hyperref}
\usepackage[headings]{fullpage}
\usepackage{enumerate}
\usepackage{mathtools}
\usepackage{tikz-cd}
\usepackage{tikz}
\usepackage{bbold}
\usepackage{graphicx} 

\theoremstyle{plain}
\newtheorem{thm}{Theorem}
\numberwithin{thm}{section}

\newtheorem{lemma}[thm]{Lemma}
\newtheorem{propn}[thm]{Proposition}
\newtheorem{cor}[thm]{Corollary}
\newtheorem{fact*}{Fact*}
\newtheorem*{thm*}{Theorem}

\theoremstyle{definition}
\newtheorem{defn}[thm]{Definition}

\newtheorem{question}[thm]{Question}

\newtheorem{rmk}[thm]{Remark}
\newtheorem{example}[thm]{Example}

\newtheoremstyle{example}
{5pt}
{0pt}
{}
{}
{\it}
{.}
{.5em}
{}
\theoremstyle{example}

\newcommand{\<}{\langle}
\renewcommand{\>}{\rangle}
\newcommand{\R}{\mathbb{R}}
\newcommand{\Z}{\mathbb{Z}}
\newcommand{\C}{\mathbb{C}}
\newcommand{\N}{\mathbb{N}}
\newcommand{\Q}{\mathbb{Q}}

\newcommand{\T}{\mathbb{T}}
\newcommand{\SSS}{\mathbb{S}}
\newcommand{\XX}{\mathbb{X}}

\newcommand{\cM}{\mathcal{M}}
\newcommand{\cA}{\mathcal{A}}
\newcommand{\cB}{\mathcal{B}}
\newcommand{\m}{\mathsf{m}}
\renewcommand{\H}{\mathcal{H}}

\newcommand{\F}{\mathcal{F}}

\newcommand{\bdd}{\mathbb{B}}
\newcommand{\dd}{\mbox{\rm d}}
\newcommand{\Borel}{{B}}
\newcommand{\supp}{\operatorname{supp}}

\newcommand{\Prob}{\operatorname{Prob}}
\newcommand{\Erg}{\operatorname{Erg}}
\newcommand{\ProbG}{\operatorname{Prob}^G}
\newcommand{\Meas}{\operatorname{Prob}_{\hspace{0.08em} 0}}
\newcommand{\MeasG}{\operatorname{Prob}_0^G}

\newcommand{\eps}{\varepsilon}

\renewcommand{\subset}{\subseteq}
\renewcommand{\supset}{\supseteq}
\newcommand{\mat}[1]{\left( \begin{smallmatrix} #1 \end{smallmatrix} \right)}
\newcommand{\id}{\mathrm{id}}
\newcommand{\1}{\mathbb{1}}

\newcommand{\Cu}{C_{\mathsf{u}}}
\newcommand{\Cc}{C_{\mathsf{c}}}
\newcommand{\Cz}{C^{}_{0}}
\newcommand{\Bap}{\mathcal{B}\hspace*{-1pt}{\mathcal{AP}}}
\newcommand{\Map}{\mathcal{M}\hspace*{-1pt}{\mathcal{AP}}}
\newcommand{\bap}{{\mathsf{BAP}}}
\newcommand{\bapx}{\mathsf{Bap}_{\cA}(X)}
\newcommand{\map}{{\mathsf{MAP}}}

\begin{document}

\title{On the spectrum of non-ergodic measures}

\author[M.~Francis]{Michael Francis}
\address{Department of Mathematics and Statistics, MacEwan University, \newline
\hspace*{\parindent}  Edmonton, Alberta, Canada}
\email{francism29@macewan.ca}
\urladdr{https://mdfrancis.github.io/}

\author[C.~Ramsey]{Christopher Ramsey}
\address{Department of Mathematics and Statistics, MacEwan University, \newline
\hspace*{\parindent}  Edmonton, Alberta, Canada}
\email{ramseyc5@macewan.ca}
\urladdr{https://sites.google.com/macewan.ca/chrisramsey/}

\author[N.~Strungaru]{Nicolae Strungaru}
\address{Department of Mathematics and Statistics, MacEwan University, \newline
\hspace*{\parindent}  Edmonton, Alberta, Canada, 
and 
\newline \hspace*{\parindent} 
Institute of Mathematics ``Simon Stoilow'', 
Bucharest, Romania}
\email{strungarun@macewan.ca}
\urladdr{https://sites.google.com/macewan.ca/nicolae-strungaru/}

\begin{abstract} Consider a topological dynamical system where the group is abelian and the topologies are locally compact and second-countable. Given an invariant measure for this system, we show that if its dynamical spectrum is contained in some Borel subset of the dual group then the same holds almost surely for all ergodic measures arising via the Choquet theorem. In particular, if the invariant measure has pure point dynamical spectrum, so do almost all the ergodic measures.  As an application, we show that given any mean almost periodic measure, in its hull there exists a Besicovitch almost periodic measure. 
\end{abstract}


\maketitle


\section{Introduction}

The discovery of quasicrystals 40 years ago \cite{She} led to a new area of mathematics called aperiodic order, an area concerned with the study of objects which show long-range aperiodic order, usually via a large Bragg diffraction spectrum. Of special interest are objects which show pure point diffraction spectrum.

Building on earlier work of \cite{bm,ARMA,Gou,Mey,Sol,Sol2}, just to name a few, pure point diffraction was characterized in terms of almost periodicity of the underlying structure in \cite{LSS}. More precisely, we have:
\begin{itemize}
\item{} pure point diffraction is equivalent to mean almost periodicity of the underlying structure; 
\item{} pure point diffraction and the so-called consistent phase property (CPP) is equivalent to Besicovitch almost periodicity of the underlying structure; 
\item{} pure point diffraction and the CPP holding along all van Hove sequences is equivalent to Weyl almost periodicity of the structure.
\end{itemize}

Similar results hold given any suitable topological dynamical system $(X, G)$,   where  $G$ is  abelian, and any pure point $G$-invariant measure $\m$   \cite{LSS2,LS}. Indeed for a transitive point $x \in X$ we have:
\begin{itemize}
    \item $x$ is mean almost periodic if and only if $x$ is generic for a $G$-invariant measure $\m$ with pure point spectrum;
        \item $x$ is Besicovitch almost periodic if and only if $x$ is generic for a $G$-invariant ergodic measure $\m$ with pure point spectrum, and the Wiener--Wintner theorem holds;
     \item $x$ is Weyl almost periodic if and only if $X$ is uniquely ergodic, has continuous eigenfunctions and pure point spectrum.   
\end{itemize}

\medskip

The above work has led to renewed interest in these spaces of measures (and functions), and especially in understanding the similarities and difference between these concepts.

One interesting observation is that all known examples of mean almost periodic functions or measures which are not Besicovitch almost periodic (see for example \cite[Remark 3.4 or Proposition A.2]{LSS}) are obtained by interlacing
two (or more) Besicovitch almost periodic functions/measures. In particular, all known examples of such functions/measures are actually Besicovitch almost periodic with respect to a different averaging sequence. This raises a natural question:

\begin{question} Given a function/measure which is mean almost periodic with respect to some averaging sequence $\cA$, can we find another averaging sequence $\cB$ such that the function/measure becomes Besicovitch almost periodic? 
\end{question}

It is the goal of the paper to show that the answer is yes. 

Our approach is as follows: starting with a mean almost periodic measure $\mu$ on $G$, we can construct a $G$-invariant probability measure $\m$ on the hull $\XX(\mu)$  
with pure point spectrum \cite{LSS,LSS2}. If we can prove that at least one ergodic measure $\omega$ on $\XX(\mu)$ has pure point spectrum, we can deduce that $\XX(\mu)$ admits Besicovitch almost periodic measures, from which the claim follows.


To this end, we prove a much more general result about $G$-invariant probability measures on topological dynamical systems. Indeed, let $(X,G)$ be a topological dynamical system with $X$ and $G$ locally compact and second-countable, and let $\m$ be a $G$-invariant probability measure on $X$. Let $\rho$ be the measure on $\Erg(X)$ which represents $\m$ via the Choquet theorem. In our main result (Theorem~\ref{thm:almosteverywhere}) we show that if $B \subseteq \widehat{G}$ is a Borel set such that the spectrum of $\m$ is contained in $B$, then for $\rho$-a.s. all $\omega \in \Erg(X)$, the spectrum of $\omega$ is also contained in $B$. As a consequence we get that if $\m$ has pure point spectrum, then so does $\omega \in \Erg(X)$ for $\rho$-a.s. $\omega$.

For brevity, we require that  all our compact and locally compact spaces/groups be Hausdorff by definition.

The paper is organized as follows: In Section~\ref{sect:prel} we review some basic definitions and properties of functions and measures. We continue in Section~\ref{sect:Borel} by reviewing properties of the Borel structure on the set of probability measures on a second-countable locally compact Hausdorff space $X$. In Section~\ref{sec:choquet} we briefly review Choquet theory, while in Section~\ref{sect:spectral} we review spectral measures and spectral decomposition. Given a dynamical system $(X,G)$ and a $G$-invariant probability measure $\m$ on $X$, Choquet theory tells us that there exists a probability measure $\rho$ on $\Erg(X)$ such that
\begin{equation}\label{eq-rep}
\int_X f(x) \, \dd \m(x) = \int_{\Erg(X)} \int_X f(y) \, \dd \omega(y) \, \dd \rho(\omega) \qquad \forall f \in \Cz(X) \,.
\end{equation}
We show in Section~\ref{sect:main} that \eqref{eq-rep} holds for all $f \in B(X)$, the algebra of bounded Borel functions. As a consequence, we get in Theorem~\ref{thm:almosteverywhere} that if the spectrum of $L^2(X,\m)$ is supported inside some Borel set $B\subseteq \widehat{G}$ then so is the spectrum of $L^2(X,\omega)$ for $\rho$-almost all $\omega \in \Erg(X)$.

As immediate consequences, we get that if $L^2(X,\m)$ has pure point or singular spectrum, then so does $L^2(X, \omega)$ for $\rho$-almost all $\omega \in \Erg(X)$. We conclude the paper by showing in Section~\ref{sec:Bap} that if $\mu \in \cM^\infty(G)$ is mean almost periodic with respect to some van Hove sequence $\cA$ then $\mu$ is Besicovitch almost periodic with respect to some translated subsequence of $\cA$.

\section{Preliminaries}\label{sect:prel}

Throughout this article, $X$ denotes a second-countable locally compact Hausdorff space. We work with several algebras of complex-valued functions on $X$. We write $\Cz(X)$ for the algebra of continuous functions vanishing at infinity. Equipped with  the uniform norm,  $\Cz(X)$ is a separable\footnote{A locally compact Hausdorff space $Y$ is  second-countable if and only if it is   $\sigma$-compact and metrizable if and only if $\Cz(Y)$ is separable  in the uniform norm.} C*-algebra. We denote by $\Cc(X)$ the norm-dense ideal in $\Cz(X)$ of compactly-supported continuous functions. If $X$ is actually compact, then $\Cz(X)=\Cc(X)=C(X)$, the unital algebra of continuous functions on $X$. We write $\Borel(X)$ for the unital algebra of  bounded Borel measurable functions on $X$. Note $\Borel(X)$ is also a  C*-algebra with respect to the uniform norm, but typically a non-separable one.

\subsection{Spaces of measures}
We work with complex Radon measures on $X$ which we  allow to be finite or infinite. A simple way to be precise about which measures  are being considered is to define them  as linear functionals.

\begin{defn}\label{def:finite}
A \textbf{finite measure} on $X$ is a  linear functional on $\Cz(X)$ that is continuous with respect to the uniform norm. We denote the space of finite measures on $X$ by $\cM_f(X)$. The \textbf{weak-star topology} on $\cM_f(X)$ is such that $\m_i \to \m$  if and only if $\m_i(f) \to \m(f)$ for all $f \in \Cz(X)$.
\end{defn}

\begin{defn}\label{def:radon}
A  \textbf{Radon measure} on $X$ is a linear functional on $\Cc(X)$ whose restriction to the ideal $\Cz(U) \subset \Cc(X)$ is uniform norm-continuous for every pre-compact open set $U \subset X$. We denote the space of Radon measures on $X$ by  $\cM(X)$. The \textbf{vague topology} on $\cM(X)$ is such that $\mu_i \to \mu$  if and only if $\mu_i(\varphi) \to \mu(\varphi)$ for all $\varphi \in \Cc(X)$.
\end{defn}

Regardless of the   linear functional formalism above, we often write  $\int_{X} f(x)\, \dd \mu(x)$ or similar  instead of $\mu(f)$.  The following remark is intended to give further clarity  about  the types of measures being considered.

\begin{rmk}
\text{ }
\begin{enumerate}[(i)]
\item Because  $\Cc(X)$ is uniform norm-dense in $\Cz(X)$, one has $\cM_f(X) \subset \cM(X)$ in a natural way.
\item Just as $\cM_f(X)$ is the continuous dual of $\Cz(X)$ in the uniform norm-topology, $\cM(X)$ is the continuous dual of $\Cc(X)$ in the  inductive limit topology 
 \cite[Theorem~C.4]{RS2}. In particular, what we call the vague topology in Definition~\ref{def:radon} could also be called a  weak-star topology.  By using two terminologies, we hope to reduce the chance of  confusion with Definition~\ref{def:finite}.
\item One may equivalently view $\cM(X)$ as the $\C$-linear span of the cone $\cM_+(X)$ of positive linear functionals on $\Cc(X)$. See again 
 \cite{RS2}.
\item Using the  appropriate version of the Riesz--Markov--Kakutani representation theorem \cite[Theorem~2.14]{rudin}, $\cM_+(X)$ may be identified with the set of (possibly infinite) positive  Borel measures $\mu$ on   $X$ that satisfy $\mu(K) < \infty$ for all compact sets $K \subset X$.  Since $X$ is assumed to be second-countable, there are no issues of regularity to be concerned about; by \cite[Theorem~2.18]{rudin} it automatically holds that\begin{align*}
\mu(E) = \inf \{ \mu(U) : U \supset E, U \text{ open}\} = \sup \{ \mu(K) : K \subset E, K \text{ compact}\}
 \end{align*}
 for every Borel set $E \subset X$. 
\item Similarly, by the appropriate version of the Riesz--Markov--Kakutani representation theorem \cite[Theorem~6.19]{rudin}, the continuous dual $\cM_f(X)=\Cz(X)^*$ may be identified with the space  of finite complex    Borel measures on  $X$. Again, since $X$ is assumed to be second-countable,   there are no issues of measure-theoretic regularity to discuss.   
\item The usual concepts of the total variation measure $|\mu|$ and Jordan decomposition of a complex measure $\mu$ make  sense not just for finite measures but also for Radon measures. See \cite[Theorem~6.5.6]{Ped} for more information.
\end{enumerate}
\end{rmk}

\subsection{Translation bounded measures}

Let $G$ be  a second-countable locally compact (and Hausdorff) abelian group.

\begin{defn} A complex Radon measure $\mu \in \cM(G)$ is called \textbf{translation bounded} if for all compact sets $K \subset G$ we have  $\sup_{t \in G} \left| \mu \right|(t+K) < \infty$.
We denote the space of translation bounded measures by $\cM^\infty(G)$. 
\end{defn}

\begin{rmk} \label{rem:2.5}\
\begin{enumerate}[(a)]  
    \item\label{Unorm} Let $U \subset G$ be a  pre-compact open set containing the identity. Then,
    \[
\| \mu \|_{U} := \sup_{t \in G} \left| \mu \right|(t+U)
\]
defines a norm on $\cM^\infty(G)$. Moreover, different choices of $U$ define equivalent norms \cite[p.~32]{bm}.
 \item Let $U \subset G$ be a  pre-compact open set containing the identity. By definition, each $\mu \in \cM(G)$ defines a continuous linear functional on $\Cz(U) \subset \Cc(G)$; the measure $\mu$ is translation bounded if and only if the functionals defined by its translates $\{ T_t \mu :t \in G\}$ are contained in some ball of the dual space  $\Cz(U)^*$. Moreover, $\| \mu \|_{U}$ is the radius of the smallest such ball. By \eqref{Unorm}, the choice of $U$ does not matter.

\item Fix $U \subset G$ a pre-compact open set and $C>0$. The Banach--Alaoglu theorem implies that 
\[
\cM_{C,U}:=\{ \mu \in \cM^\infty(G) : \| \mu \|_{U}  \leq C \}
\]
is a vaguely compact subset of $\cM^\infty(G)$. Since $G$ is second-countable, the vague topology is metrizable on $\cM_{C,U}$ \cite[Theorem 2]{BL}.
\item A measure $\mu \in \cM(G)$ is translation bounded if and only if  $\mu*\varphi \in \Cu(G)$ for all $\varphi \in \Cc(G)$, where $\Cu(G)$ denotes the algebra of bounded uniformly continuous functions. See \cite[Theorem~1.1]{ARMA1} or \cite[Prop.~4.9.21]{MoSt}.
\end{enumerate}
\end{rmk}

\begin{defn}
The   \textbf{hull} of a translation bounded measure $\mu \in \cM^\infty(G)$  is 
\[
\XX(\mu):= \overline{\{ T_t \mu :t \in G\} \, ,  }
\]
the vague-closure  of the orbit of $\mu$ under  the natural translation action of $G$.
\end{defn}

\begin{rmk}
For  any pre-compact open set containing the identity  $U \subset G$, one has $\XX(\mu) \subset \cM_{C,U}$ for $C:= \| \mu \|_{U}$. Since $\cM_{C,U}$  is  vaguely compact and metrizable, the hull $\XX(\mu)$ is also a compact metrizable topological dynamical system. 
\end{rmk}

\subsection{Almost periodic functions}

We briefly review various notions of almost periodicity for functions and measures, particularly  Besicovitch and mean almost  periodicity. For a more detailed review, see \cite{LSS}.  As before, $G$ is a second-countable locally compact abelian group. We write $\Cu(G)$ for the algebra of bounded uniformly continuous functions on $G$ in the uniform norm. Let us first recall the strongest and most classical notion: a function $f \in \Cu(G)$ is called \textbf{Bohr almost periodic} if for each $\eps>0$ the set 
\[
P_{\eps,\sup} := \{ t \in G : \|T_t f-f \|_\infty < \eps \}
\]
of $\eps$-almost periods is relatively dense (also called co-compact) in $G$. This is equivalent to the Bochner condition that the orbit $\{T_t f :t \in G \}$ has compact closure in $(\Cu(G), \|\,\|_\infty)$. One may also characterize the Bohr almost periodic functions as the closure in the uniform norm of the algebra of trig polynomials on $G$.

\begin{rmk}
Both the  Bohr almost periodic functions and  the collection of all uniformly continuous functions on $G$ are  C*-algebras in the uniform norm. The Gelfand spectrum of the Bohr almost periodic functions is known as the Bohr compactification  and has been intensely studied. The spectrum of the uniformly continuous functions has also been studied and is sometimes referred to as the Samuel compactification or ``universal ambit'' \cite{samuel}.
\end{rmk}

To talk about mean and Besicovitch almost periodicity, we first need to introduce nice averaging sequences. For functions, the so-called F{\o}lner condition suffices.  When dealing with measures, however, one needs to work with more restrictive sequences  called van Hove sequences.

\begin{defn} A sequence of pre-compact sets  $F_n \subset G$ is called a \textbf{F{\o}lner sequence} if for all $t \in G$ we have 
\[
\lim_n \frac{|F_n \triangle (t+F_n)|}{|F_n|}=0 \,,
\]
where $\triangle$ denotes the symmetric difference of sets and the vertical bars denote a fixed Haar measure on $G$. Sometimes (when dealing with the Birkhoff ergodic theorem) we need to restrict attention to \textbf{tempered} F{\o}lner sequences, meaning  F{\o}lner sequences $F_n$ with the additional property that there exists some $C>0$ such that for all $n$ we have 
\[
\left| \bigcup_{k=1}^{n-1} F_n-F_k \right| < C |F_n| \,.
\]
\end{defn}
\begin{defn}
A sequence of pre-compact sets $A_n \subset G$ is called a \textbf{van Hove sequence} if for each compact set $K \subset G$ we have 
\[
\lim_n \frac{|\partial^K(A_n)|}{|A_n|} =0 \,.
\]
Here, the \textbf{$\mathbf{K}$-boundary} $\partial^K(A_n)$ of $A_n$ is defined as follows:
\[
\partial^K(A_n) \coloneqq \bigl( \overline{A_n+K}\setminus A_n^\circ\bigr) \cup
\bigl((\left(\overline{G \backslash A_n}\right) - K)\cap \overline{A_n}\, \bigr)
\,.
\]

\end{defn}

\begin{rmk}\
\begin{enumerate}[(a)]
\item Any F{\o}lner sequence has a tempered subsequence  \cite[Prop.~1.4]{Lin}. 
\item  Any van Hove sequence is trivially a F{\o}lner sequence. The converse is not true. 
    \item  A locally compact abelian group  admits a F{\o}lner sequence if and only if it admits a van Hove sequence if and only if the group is  $\sigma$-compact \cite{SS}. Note that, since we are dealing with abelian groups, the usual complications with amenability are not an issue.
    \item  If $A_n$ is a F{\o}lner/van Hove sequence and $t_n \in G$ then $t_n+A_n$ is also a F{\o}lner/van Hove sequence. 
\item These definitions do not require averaging sequences to exhaust the group. So, for example, $A_n = [0,n]$ is an acceptable F{\o}lner sequence for $G=\R$. 
\end{enumerate}
\end{rmk}

\smallskip

We can now introduce the Besicovitch semi-norm and mean/Besicovitch almost periodicity. Let us note here that all these concepts can be defined more generally for functions which are locally integrable, compare \cite{LSS}. Since below we will always deal with functions in $\Cu(G)$, we restrict our definitions to this space.

\begin{defn}
Each F{\o}lner sequence  $\cA= \{A_n\}$  in $G$ defines a semi-norm on $\Cu(G)$ via
\[
\| f \|_{\mathsf{b},1,\cA}:= \limsup_n \frac{1}{|A_n|} \int_{A_n} |f(t)| \, \dd t \,.
\]
This is called the \textbf{Besicovitch semi-norm} associated to $\cA$.
\end{defn}

\begin{defn} Let $f \in \Cu(G)$ and let $\cA$ be a F{\o}lner sequence in $G$. We say that $f$ is \textbf{mean almost periodic} with respect to $\cA$ if for each $\eps >0$ the set 
\[
P_{\eps,\mathsf{b}}:= \{ t \in G : \| T_t f-f \|_{\mathsf{b},1,\cA} < \eps \}
\]
is relatively dense. We say  $f \in \Cu(G)$ is \textbf{Besicovitch almost periodic}  with respect to $\cA$ if for each $\eps>0$ there exists some Bohr almost periodic function $g$ such that 
\[
\|f-g\|_{\mathsf{b},1,\cA} < \eps \,.
\]
We denote the spaces of mean and Besicovitch almost periodic functions by $\map_{\cA}(G)$ and $\bap_{\cA}(G)$. One has $\bap_{\cA}(G) \subseteq \map_{\cA}(G)$ and the inclusion is in general strict.
\end{defn}

\begin{rmk}\text{ }
\begin{itemize}
\item[(a)] Note  that in the related work \cite{LSS},  slightly larger spaces of functions  denoted 
$\mathsf{Bap}_\cA(G)$ and $\mathsf{Map}_\cA(G)$ are used whose constituents are not required to be uniformly continuous. Intersecting the latter spaces with $\Cu(G)$ gives the spaces  $\bap_\cA(G)$ and  $\map_\cA(G)$ used here.
    \item[(b)] One can define a Besicovitch semi-norm and mean/Besicovitch almost periodicity for each $1 \leq p <\infty$ (see \cite{LSS}). While these concepts can be different in general for different values of $p$, on the space $\Cu(G)$ the concept of mean/Besicovitch almost periodic function is the same for all $1 \leq p < \infty$. For this reason we restrict to the case $p=1$ in this paper.
    \item[(c)] $\bap_\cA(G)$ and  $\map_\cA(G)$ are norm-closed in $\Cu(G)$ and so are C*-algebras as well \cite{LSS2}.
\end{itemize}
\end{rmk}

The notions of mean and Besicovitch almost periodicity  can be  extended from the setting of functions to the setting of  translation bounded measures. These classes of measures are important for the classification of systems with pure point spectrum \cite{LSS,LSS2}.

\begin{defn} Suppose that  $\mu \in \cM^\infty(G)$. We say that $\mu$ is \textbf{mean/Besicovitch almost periodic} with respect to a van Hove sequence $\cA$ if for all $\varphi \in \Cc(G)$ the function $\mu*\varphi \in \Cu(G)$ is mean/Besicovitch almost periodic with respect to $\cA$. The spaces of mean and almost periodic Besicovitch measures are denoted by $\Map_{\cA}(G)$ and $\Bap_{\cA}(G)$, respectively.
\end{defn}

The following technical lemmas show that, in order to establish Besicovitch almost periodicity of a measure, one  does not really need to consider its convolution against  every continuous  compactly-supported function;  a  suitable  dense set is good enough. A similar statement holds for mean almost periodic measures, but we will not need this.

\begin{lemma}\label{lem-D-exists} 
Let $X$ be a second-countable locally compact Hausdorff space. Then, there exists some countable set $D \subseteq \Cc(X)$ such that, for each $\varphi \in \Cc(X)$, there exists a sequence $\psi_n \in D$ that converges to $\varphi$ uniformly and whose supports are contained in a single compact set $K_\varphi \subset G$.
\end{lemma}
\begin{proof}
Let  $U_n \subseteq X$ be an increasing sequence of pre-compact open sets that covers $X$.
For each $n$, let $D_n$ be a countable dense subset of $\Cz(U_n)$. Putting $D = \bigcup_n D_n$, we get the claim. 
\end{proof}

\begin{lemma}\label{lemma-D-used} Let $G$ be second-countable and let $D$ be as in Lemma~\ref{lem-D-exists}. Let $\mu \in \cM^\infty(G)$. Then $\mu \in \Bap_{\cA}(G)$ if and only if $
\mu*\varphi \in \bap_{\cA}(G)   
$ for all $\varphi \in D   
$.

\end{lemma}
\begin{proof}
First,  observe that if  $\mu \in \cM^\infty(G)$ and $\varphi \in \Cc(G)$ has $\supp(\mu) \subseteq -K$ for compact $K \subset G$, then 
\[ 
\| \mu*\varphi \|_{\mathsf{b},1,\cA} \leq \| \mu*\varphi \|_{\infty} \leq \| \varphi \|_{\infty} \| \mu \|_{K} \, . \]
Now fix $\varphi \in \Cc(G)$ and let $\varphi_n \in D$ converge uniformly to $\varphi$ and have supports contained in a compact set $K \subset G$. By the above observation,  $\mu * \varphi_n \to \mu * \varphi$ in $\| \cdot \|_{\mathsf{b},1,\cA}$. Since $\mu  * \varphi_n  \in  \bap_{\cA}(G)$, 
there exist Bohr almost periodic functions $f_n$ such that $\| \mu*\varphi_n-f_n \|_{\mathsf{b},1,\cA} < \frac{1}{n}$. Then,
\[
\| \mu*\varphi-f_n \|_{\mathsf{b},1,\cA}  \leq\| \mu*\varphi-\mu*\varphi_n \|_{\mathsf{b},1,\cA} +\| \mu*\varphi_n-f_n \|_{\mathsf{b},1,\cA}  \to 0 \,.
\]
\end{proof}

\subsection{Invariant measures on the hull}

We conclude this preliminary section by  recalling the following result from \cite{LS2}. Compare also  \cite{LS}.

\begin{thm}[Proposition 4.5(b) \cite{LS2}] \label{thm:map-gen}
Let $G$ be a second-countable locally compact abelian group and let $\cA$ be a van Hove sequence in $G$. Fix $\mu \in \cM^\infty(G)$. Then,  there exists  some subsequence $\cB=\{ B_n \}$ of $\cA$ and a $G$-invariant probability measure $\m$ on $\XX(\mu)$ such that for all $F \in C(\XX(\mu))$ we have
\[
\lim_n \frac{1}{|B_n|} \int_{B_n} F(T_t \mu) \, \dd t = \int_{\XX(\mu)} F(\omega) \, \dd \m(\omega) \,.
\]
Furthermore, if $\mu \in \Map_{\cA}(G)$ then $\m$ has pure point dynamical spectrum.\qed 
\end{thm}

\section{The Borel structure of probability  measures}\label{sect:Borel}
In this section we collect a series of needed results  concerning the Borel structure of spaces of probability measures. As usual, $X$ denotes a second-countable locally compact Hausdorff space.

Recall that, if $Z$ is a Banach space, the Banach--Alaoglu theorem gives that the unit ball of $Z^*$ is a compact   space in the weak-star topology and, if $Z$ is furthermore separable, then the ball of $Z^*$ is second-countable (or, equivalently, metrizable). In particular, the unit ball of $\cM_f(X)\coloneqq \Cz(X)^*$ is a compact metrizable space. The space of quasi-probability measures $\Meas(X)$ is  the positive part of this unit ball, i.e. the convex compact metrizable space of positive Borel measures $\mu$  with $\|\mu\| =\mu(X) \leq 1$.   The space of probability measures $\Prob(X) \subset \Meas(X)$ is the convex subspace of $\mu \in \Meas(X)$ such that $\|\mu\|=\mu(X)=1$, in other words the positive part of the unit sphere of  $\Cz(X)^*$. The difference between $\Prob(X)$ and $\Meas(X)$ spaces is fairly minor;  $\Meas(X)$ is the convex hull of $\Prob(X)$ and the zero measure. If $X$ is compact, $\Prob(X)$ is itself weak-star compact but, in general,  $\Prob(X)$ is only a $G_\delta$ subset of $\Meas(X)$. 

Recall that  $\Borel(X)$ denotes the space of bounded Borel measurable functions on $X$ equipped with  the sup-norm. Each $f \in \Borel(X)$, induces an
 affine function  $\Phi_f : \Meas(X) \to \C$ by \[
\Phi_f (\omega) := \int_X f \, \dd \omega \,.
 \] 
Clearly the assignment $f\mapsto \Phi_f$ is isometric for the uniform norms: \[
\| \Phi_f \|_\infty = \|f\|_\infty \,.
\]
If $f \in \Cz(X)$, then $\Phi_f \in C(\Meas(X))$ by definition of the weak-star topology. Although not immediately obvious, it is also true that, if $f \in \Borel(X)$, then   $\Phi_f \in \Borel(\Meas(X))$.  The latter is a consequence of the proof of   Lemma~2.3 in \cite{varadarajan}. For the reader's convenience, we give a proof here.

\begin{propn}\label{measurability}
Let $X$ be a second-countable locally compact  space. Then, the following sigma-algebras on $\Meas(X)$ are equal:
\begin{enumerate}[(i)]
\item the Borel sigma-algebra generated by the compact metrizable topology that $\Meas(X)$ inherits from the  weak-star topology on $\Cz(X)^*$,
\item the ``customary sigma-algebra'' of \cite{varadarajan}. By definition, this is  the smallest sigma-algebra on $\Meas(X)$ for which  $\omega \mapsto \omega(A)$ is a measurable function for every Borel set $A \subset X$. 
\end{enumerate}
Moreover, this sigma-algebra is such that, for every $f \in \Borel(X)$, the function $\Phi_f: \Meas(X) \to \C$ defined by $\Phi_f(\omega) := \int_X f \, \dd \omega$ is measurable (obviously $\Phi_f$ is affine and $\|\Phi_f\|_\infty=\|f\|_\infty$).
\end{propn}
We remark that this proposition in particular shows that the Borel structure of $\Meas(X)$ only depends on the Borel structure of $X$ and not on its topology.
\begin{proof}
First we show (ii) is contained in (i).  

Let $\F$ denote the set of all $f \in \Borel(X)$ for which  $\Phi_f \in \Borel(\Meas(X))$. If $f \in \Cz(X)$, then $\Phi_f \in C(\Meas(X))$ by definition of the weak-star topology, so  
\[
\Cz(X) \subset \F \subset \Borel(X)\,.
\]
Obviously $\F$ is closed under taking $\C$-linear combinations. We claim $\F$ is also closed under uniformly bounded pointwise limits. Indeed, suppose that $(f_n)$ is a uniformly bounded sequence in $\F$ converging pointwise to $f \in \Borel(X)$. Then, by the dominated convergence theorem, for all $\omega \in \Meas(X)$ we have 
\[
\int_X f_n \, \dd \omega \to \int_X f \, \dd \omega \,.
\]
This means that $\Phi_{f_n}$ converges pointwise to $\Phi_f$.  This implies that $\Phi_{f} \in \Borel(\Meas(X))$ and hence $f\in \F$. 

Let $\mathcal{D}$ denote the set of Borel subsets $D \subset X$ for which $\1_D \in \F$. Observe that $\mathcal{D}$ is a $\lambda$-system\footnote{By definition, this means: (i) $X \in \mathcal{D}$, (ii) if $D_1, D_2 \in \mathcal{D}$ and $D_1 \subset D_2$ then $D_2 \setminus D_1 \in \mathcal{D}$, and (iii) if $D_1 \subset D_2 \subset \ldots$ all belong to $\mathcal{D}$, then $\bigcup_n D_n \in \mathcal{D}$.}. Every compact set $K \subset X$ is a $G_\delta$ set and, moreover, $\1_K$ may  be expressed as the pointwise limit of a sequence of continuous functions $f_n \in \Cz(X)$. Thus, $\1_K \in \F$ and $K \in \mathcal{D}$ for all compact sets $K \subset X$. Now, compact subsets of $X$ are a $\pi$-system\footnote{A nonempty collection of sets that is closed under intersection.}, so  Sierpi\'{n}ski--Dynkin's $\lambda$-$\pi$ theorem implies that $\mathcal{D}$ contains the sigma-algebra on $X$ generated by the compact sets. Thus,  $\mathcal{D}$ is the Borel sigma-algebra of $X$ and  the containment of (ii) in (i) is established.  

For the ``moreover'' statement, note that, by linearity, we have  that $\mathcal{F}$ contains all the simple functions so, since $\F$  is closed under uniformly bounded pointwise limits, it follows that $\F = \Borel(X)$.

It remains to show (i) is contained in (ii) i.e. that the customary sigma-algebra contains all the weak-star open subsets of $\Meas(X)$. Crucially, since $\Meas(X)$ is second-countable,  it is enough to check that the  open sets belonging to the usual subbase for the weak-star topology belong to the customary sigma-algebra (see Lemma~\ref{subbaselemma} below). That is, it suffices to check that, given $f\in \Cz(X)$ and  $V \subset \C$ open, the subbasic open set 
\[
\{ \omega \in \Meas(X) : \int_X f  \, \dd \omega  \in V\}
\]
is measurable for the customary sigma-algebra, i.e. $\omega \mapsto \int_X f \, \dd \omega$ is measurable for the customary sigma-algebra. Indeed, this holds if $f$ is instead a simple function by definition of the customary sigma-algebra. Since every $f\in \Borel(X)$, and in particular every $f \in \Cz(X)$,  is the uniform limit of a sequence of simple functions and measurable functions are closed under uniform limits (even pointwise limits), we get the desired conclusion.
\end{proof}

\begin{lemma}\label{subbaselemma}
Let $Y$ be a set and let $\mathscr{S}$ be collection of subsets of $Y$ such that the topology $\tau$ generated by $\mathscr{S}$ is second-countable. Then, the sigma-algebra   $\Sigma(\mathscr{S})$ generated by $\mathscr{S}$ coincides with the sigma-algebra $\Sigma(\tau)$ generated by $\tau$.
\end{lemma}
\begin{proof}
Obviously $\Sigma(\mathscr{S}) \subset  \Sigma(\tau)$. The basis $\mathscr{B}$ for the topology $\tau$ consisting of intersections of finite subcollections of  $\mathscr{S}$ satisfies $\mathscr{B} \subset \Sigma(\mathscr{S})$. Second-countability of $\tau$ implies there is a countable basis $\mathscr{B}' \subset \mathscr{B}$. Every open set in $\tau$ is then a countable union of elements of $\mathscr{B}' \subset \Sigma(\mathscr{S})$ demonstrating that $\tau \subset \Sigma(\mathscr{S})$, whence $\Sigma(\tau) \subset \Sigma(\mathscr{S})$.
\end{proof}

Note the second-countability hypothesis cannot be dropped in the above lemma. Consider for example the collection of singleton subsets of an uncountable set. The sigma-algebra generated by the singletons is the countable-cocountable sigma-algebra. The topology generated by the singletons is the discrete topology.

There is another sigma-algebra one could consider putting on $\Meas(X)$. Namely, the Borel sigma-algebra obtained by viewing $\Meas(X)$ as a subspace of $\Borel(X)^*$ in the weak-star topology. This sigma-algebra is generally different.

\begin{example}
Let $X=[0,1]$. Then, the weak-star topology  that $\Meas(X)$ inherits from $\Borel(X)^*$ generates a larger Borel sigma-algebra than the weak-star topology that $\Meas(X)$ inherits from $C(X)^*$. To see this, note that the latter sigma-algebra consists of continuum-many sets. Indeed, a transfinite induction argument shows that any sigma-algebra that can be generated by continuum-many sets or fewer contains at most continuum-many sets. On the other hand, for any subset $S \subset [0,1]$ whatsoever, one has that the set $U_S \subset \Meas(X)$ of  measures $\omega$ such that $\omega(\{x\}) >0$ for at least one $x \in S$ is open for the weak-star topology $\Meas(X)$ inherits from $\Borel(X)^*$. Thus, this topology and the sigma-algebra it generates contain at least as many sets as the power set of $[0,1]$.
\end{example}

The following well-known lemma is one way to give rigorous sense to the assertion that the algebra of bounded Borel functions is ``generated'' by the algebra of continuous functions. For convenience, we include a proof.

\begin{lemma}\label{CgenB}
Let $X$ be a second-countable locally compact  space. Then, $\Borel(X)$ is the smallest algebra containing $\Cz(X)$ that is closed under taking pointwise limits of uniformly bounded sequences.
\end{lemma}
\begin{proof}
Note that $\Borel(X)$ is an algebra containing $\Cz(X)$  closed under taking pointwise limits of  uniformly bounded sequences and that the intersection of a collection of such algebras is another such algebra. It 
follows that a smallest such algebra $\F$ exists and that $\Cz(X) \subset \F \subset \Borel(X)$. Firstly, for  every compact set $K \subset X$, the indicator function $\1_K$ belongs to $\F$ since it can be expressed as the pointwise limit of a uniformly bounded sequence of continuous functions. Next, we argue that the collection of all Borel sets whose characteristic functions  belong to $\F$ is a sigma-algebra. Indeed, recalling $X$ is second-countable, $\1_X \in \F$. If $\1_A \in \F$, then   $\1_{X \setminus A} =\1_X- \1_A \in \F_A$. If $\1_A, \1_B \in \F$, then $\1_{A \cap B} = \1_A \1_B \in \F$.  If $\1_{A_n} \in \F$ for $n=1,2,3,\ldots$ then $\1_A \in \F$ for $A=\bigcap_{n=1}^\infty A_n$ because 
the  pointwise limit of the uniformly bounded sequence $\1_{A_1} , \1_{A_1 \cap A_2}, \1_{A_1\cap A_2 \cap A_3}, \ldots$ is $\1_A$. Compact subsets of $X$ generate the whole Borel sigma algebra, so all indicator functions of Borel sets belong to $\F$. Finally, simple functions are norm-dense in $\Borel(X)$, so the result follows.
\end{proof}

We can now prove the following result which is important for our argument. Proposition~\ref{measurability}  ensures that the right hand side of the equation in the below proposition is meaningful when $f \in \Borel(X)$.

\begin{propn}\label{C(X)toB(X)}
Let $X$ be a second-countable locally compact  space. Fix $\m \in \Meas(X)$. Suppose that $\rho \in \Prob(\Meas(X))$ represents $\m$ in the sense that
\begin{align*}
\int_X f(y) \, \dd \m(y) = \int_{\Meas(X)} \int_X f(x) \, \dd \omega(x) \, \dd \rho(\omega)  && \text{ for all } f \in \Cz(X). 
\end{align*}
Then, the same equation furthermore holds for all $f \in \Borel(X)$. 
\end{propn}
\begin{proof}
First, Proposition~\ref{measurability} ensures that $\omega \mapsto \int_X f \, \dd  \omega$ is a bounded Borel function on $\Meas(X)$ for every $f \in \Borel(X)$. Thus, the right hand side of the desired equation at least makes sense for $f \in \Borel(X)$. Now, let $\mathcal{F} \subset \Borel(X)$  denote the collection of functions $f$ for which the desired equality holds. By assumption, $\Cz(X) \subset \mathcal{F}$. Using the dominated convergence theorem, it is also easy to see that $\mathcal{F}$ is closed under uniformly bounded pointwise limits.  Thus, by Lemma~\ref{CgenB}, $\F = \Borel(X)$. 
\end{proof}

\section{Review of Choquet theory}\label{sec:choquet}

Choquet theory is concerned with expressing the points of a compact convex subset of  a locally convex topological vector space as barycentres of probability measures on the set of extreme points. A major application, and the one which is relevant here, is   the  decomposition of invariant measures as weighted combinations of ergodic measures. We provide a brief of review of the relevant aspects of this theory.

\begin{thm}[Choquet theorem]\label{choquetthm}
Let $K$ be a compact convex metrizable subset of a (Hausdorff) locally convex vector space $V$. Then, for each $z_0 \in K$, there exists $\rho \in \Prob(K)$   supported in the set $E \subset K$ of extreme points  such that 
$\Phi(z_0) = \int_E \Phi(z) \, \dd  \rho(z)$  for every continuous affine function $\Phi$ on $K$. \qed
\end{thm}

\begin{rmk}
In general, it may not be the case that the measure $\rho$ is unique. This already fails when $K$ is a square in $\R^2$. When uniqueness does hold, $K$ is referred to as a \textbf{Choquet simplex}.
\end{rmk}

The reader may refer to the monograph \cite{phelps} for a proof of Theorem~\ref{choquetthm} (see \S3 of \cite{phelps}) as well as an  exposition of the wider theory. Assuming, as we did, that $K$ is metrizable,  it is easy to show the set of extreme points $E$ is a $G_\delta$ subset of $K$ (see \cite{phelps}, Proposition~1.3) so that ``supported in'' makes sense above. 
In the nonmetrizable case, $E$ need not even be a Borel subset of $K$ and  additional work is needed to clarify what it means for a measure on $K$ to be    ``supported in $E$''. This extension of the Choquet theorem to the nonmetrizable case is usually referred to as the Choquet--Bishop--de Leeuw theorem. We will not need this extra generality.

The Choquet theorem has the following immediate consequence concerning the expression of  probability measures as barycentres.

\begin{cor}\label{choqformeas}
Let $X$ be a second-countable locally compact  space. Let $K \subset \Meas(X)$ be a weak-star closed  convex set. Then, for each $\m \in K$, there exists $\rho \in \Prob(K)$ supported  in the set $E \subset K$ of extreme points such that 
\[
\int_X f(x) \, \dd \m(x) = \int_E \int_X f(x) \, \dd  \omega(x) \, \dd  \rho(\omega) 
\]
for all $f \in \Cz(X)$. \qed
\end{cor}

\begin{rmk}
If $K$ is a compact convex subset of a locally convex topological vector space $V$, the most obvious supply of continuous affine functions on $K$ are the restrictions to $K$ of continuous affine functions on $V$, which is to say constant   shifts of continuous linear functionals.  In general, not all continuous affine functions on $K$ need to be of this form. Indeed, this can fail even in the setting of Corollary~\ref{choqformeas} above. The weak-star topology on  $\Cz(X)^*$  has dual $\Cz(X)$, but not every weak-star continuous affine function  on $K$ is necessarily of the form $\omega \mapsto c+\int_X f \, \dd \omega$ for some $f \in \Cz(X)$, $c \in \C$. See Example~\ref{affineex} below. Thus, Corollary~\ref{choqformeas} as we stated it did not leverage Theorem~\ref{choquetthm} to its fullest possible extent.   That being said,  specifying $\varphi(z_0)$ for all $\varphi \in V^*$ does of course determine $z_0 \in K$ uniquely, simply because $V^*$  separates the points of $V$.
\end{rmk}

\begin{example}\label{affineex}
Let $X=\N$. Let $K \subset \Meas(X)$ be the set of    $\omega \in \Meas(X)$   that satisfy 
\[
\omega(\{n\}) \leq \frac{1}{n2^n} \qquad \forall n \in \N \,.
\]
Then, $K$ is a closed convex subset of $\Meas(X)$ and $F(\omega) = \sum  n \omega(\{n\})$ defines a continuous affine function $F$ on $K$ that does not arise as integration against any element of $\Cz(X)$.
 \end{example}

Let  $X$ be a second-countable locally compact  space and let $S$ be some collection of homeomorphisms $X \to X$. Then,  the set $\Meas^S(X)$ of $\omega \in \Meas(X)$ satisfying $\theta_*(\omega)=\omega$ for all $\theta \in S$ is a closed convex subset of $\Meas(X)$. In particular, if $G$ is a group\footnote{Later we will take $G$ to be a topological group, but the topology of $G$ plays no role in the present discussion.} acting by homeomorphisms on $X$, then the set  $\MeasG(X)$ of $G$-invariant measures is a closed convex subset of $\Meas(X)$.

The set $\ProbG(X)$ of $G$-invariant probability measures on $X$ is a convex subset of $\MeasG(X)$. The two spaces are quite closely related;  $\MeasG(X)$ is the convex hull of $\ProbG(X)$ and the zero measure. If  $X$ is compact, then $\ProbG(X)$ is closed but,  in general, we can only say that $\ProbG(X)$ is a $G_\delta$  subset of $\MeasG(X)$. The set of extreme points of $\ProbG(X)$ is exactly the set $\Erg(X)$ of ergodic probability measures for the action (see \cite{bekka-mayer}, Proposition~3.1). The set of extreme points of $\MeasG(X)$ is $\Erg(X)\cup\{0\}$.

Concerning the expression of invariant probability measures as barycentres of ergodic probability measures, Choquet theory gives the following:

\begin{cor}
Let $X$ be a second-countable locally compact  space. Let $G$ be  group acting by homeomorphisms  on $X$. Let $\m \in \ProbG(X)$. Then, there exists a probability measure $\rho$ on $\Erg(X)$  such that for all $f \in \Cz(X)$ we have
\[
\int_X f(x) \, \dd \m(x) = \int_{\Erg(X)} \int_X f(x) \, \dd \omega(x) \, \dd \rho(\omega)\,.
\]
 \qed
\end{cor}
\begin{proof}
If $X$ is compact, so that  $\ProbG(X)$ is a closed convex subset of $\Meas(X)$, this is a direct consequence of Corollary~\ref{choqformeas}. If $X$ is non-compact, we instead view $\m$ as lying in the closed convex set $\MeasG(X) \subset \Meas(X)$ and express  $\m$ as the barycentre of a probability measure $\rho$ on   $\Erg(X) \cup \{0\}$. Because $\m$ is a probability measure, it is easy to see $\rho(\{0\})=0$ so that $\rho$ is actually supported in $\Erg(X)$. 
\end{proof}

\begin{rmk}
Though peripheral to our needs, we mention that the measure $\rho$ on $\Erg(X)$ representing a given $\m \in \ProbG(X)$ is actually unique, i.e. the spaces $\Prob(\Erg(X))$ and $\ProbG(X)$ are in 1-1 correspondence. More briefly, at least in the case where $X$ is compact, one could say ``$\ProbG(X)$ is a Choquet simplex'' (usually Choquet simplices are required to be compact by definition).
\end{rmk}

\section{Spectral decomposition}\label{sect:spectral}

In this section, $G$ denotes a second-countable locally compact abelian group with a fixed Haar measure and $\widehat G$ denotes its Pontryagin dual. 
\subsection{Spectral measures and the spectral theorem} Suppose $U : G \to \bdd(\H)$ is a strongly-continuous unitary representation  of    $G$ on a Hilbert space $\H$. We review the constructions of the following associated objects, primarily with the aim of fixing terminology and notation:

\begin{itemize}
\item \emph{Spectral measures:} each $h \in \H$ determines  a positive Borel measure  $\sigma_h$ on $\widehat G$.
\item \emph{Projection-valued measure:} a map $B \mapsto P(B)$ from Borel subsets of $\widehat G$ to projections on $\H$.
\end{itemize}

One way to define $\sigma_h$ is as the unique positive measure satisfying: 
\begin{align}
\int_{\widehat G}  \chi(t) \, \dd \sigma_{h}(\chi) = \<U(t)h,h\> \text{ for all } t \in G,
\end{align}
using Bochner's theorem to justify the fact that these integrations  determine $\sigma_h$ uniquely.
A second (equivalent) approach to spectral measures, which we will make use of here, is to first pass from $U$ to its associated $*$-representation $\pi : \Cz(\widehat G) \to \bdd(\H)$ and define  $\sigma_h$ as the unique positive measure satisfying:
\begin{align}
\int_{\widehat G}  f  \, \dd \sigma_{h} = \<\pi(f)h,h\> \text{ for all }f \in \Cz(\widehat G),
\end{align}
appealing to the  Riesz--Markov--Kakutani representation theorem instead of Bochner's theorem. This second approach is conventional in operator theory and additional details can be found in a variety of sources such as \cite{folland,pedersen,williams}. We summarize the main relevant points below.

\begin{defn}
Given a strongly-continuous unitary representation of $U$ of $G$ on a Hilbert space $\H$, the \textbf{associated $*$-representation} $\pi : \Cz(\widehat G) \to \bdd(\H)$ is defined by 
\[
\pi(\widehat f) h = \int_G f(t) U(t) h  \, \dd t
\]
for $f \in \Cc(G)$, $h \in \H$ and extended  by continuity to $\Cz(\widehat G)$. Here, $\widehat f \in \Cz(\widehat G)$ denotes the Fourier transform of $f$ which we take to be 
\[
\widehat f(\chi) = \int_G f(t) \overline{\chi(t)} \, \dd t \,.
\]
\end{defn}

\begin{example}
Suppose above that  $G=\Z$ and make the identification  $\widehat G = \T$. Then there   is a single unitary $u \in \bdd(\H)$ such that $U(n)=u^n$ for all $n \in \Z$. The associated $*$-representation $\pi$ is  the usual continuous functional calculus: $f \mapsto f(u) : C(\T) \to \bdd(\H)$.
\end{example}

\begin{defn}\label{specY}
Let $Y$ be a second-countable locally compact  space. Let $\pi : \Cz(Y) \to \bdd(\H)$ be a nondegenerate $*$-representation. 
\begin{enumerate}[(a)]
\item Given $h \in \H$, the associated \textbf{spectral measure}  is the  positive Borel measure $\sigma_h$ on $Y$ uniquely determined by 
\[
\int_Y f(y) \, \dd\sigma_h (y) = \<\pi(f)h,h\> \qquad \forall f  \in \Cz( Y) \,.
\]
\item  The \textbf{canonical Borel extension} of $\pi$ is the $*$-representation $\overline \pi : \Borel(Y) \to \bdd(\H)$ of the larger 
algebra of bounded Borel functions on $Y$ such that, for all $f \in \Borel(Y)$, $\overline \pi(f)$ is the unique operator satisfying 
\[
\< \overline \pi(f)h,h \> = \int f(y) \, \dd\sigma_h(y) \qquad \forall h \in \H\,.
\]
See \cite[Proposition IX.1.12]{conway}.
\item For every Borel set $B \subset X$, the \textbf{spectral projection} is defined as $P(B) \coloneqq \overline \pi (\1_B)$ where $\1_B$ is the indicator function of $B$. 
\end{enumerate}
\end{defn}

\begin{example}\label{multopex}
Suppose above that $\H=L^2(\Omega,\mu)$ where $(\Omega,\Sigma,\mu)$ is a measure space and that there is a measurable map  $p:\Omega \to Y$ such that $\pi$ is the resulting multiplication representation: $\pi(f)h = (f \circ p) h$ for $f \in \Cz(Y)$, $h \in L^2(\Omega,\mu)$. Then:
\begin{itemize}
\item Given $h \in L^2(\Omega,\mu)$, the associated spectral measure  is the pushforward $\sigma_h = p_*(|h|^2 \mu)$. 
\item The canonical Borel extension is given by $\overline \pi (f) h =(f \circ p) h$ for $f \in \Borel(Y)$, $h \in L^2(\Omega,\mu)$.
\item For  $B \subset X$ and $h \in L^2(\Omega,m)$, we have $P(B)h= \1_{p^{-1}(B)} h$.
\end{itemize}
\end{example}

We note some simple relationships between the spectral measures and the spectral projections.
\begin{lemma}
In the setting of Definition~\ref{specY} above, let $B \subset Y$ be Borel and  fix $h \in \H$. Then,
\begin{enumerate}[(a)]
\item $\|P(B)h\|^2 = \sigma_h(B)$
\item $\sigma_h |^{\phantom{B}}_B = \sigma_{P(B)h}$.
\end{enumerate}
\end{lemma}
\begin{proof}
Note that
\[
\|P(B)h\|^2 = \< P(B)h,h\> = \int_Y \1_B \sigma_h = \sigma_h(B) \,,
\]
and that,  for  any $f\in C(Y)$, we have 
\[
\int_Y f \, \dd \sigma_{P(B)h} = \< \pi(f) \overline\pi(\1_B) h, \overline \pi(\1_B)h\> = \< \overline \pi (f \cdot \1_B) h,h\> = \int_Y f \cdot \1_B \, \dd \sigma_h = \int_B f \, \dd \sigma_h\,.
\]
\end{proof}

\begin{rmk}\label{supprmk}
Condition~(a) above shows that the family of spectral measures $\{\sigma_h :h \in \H\}$ and the projection-valued measure $P$ of  $\pi$ determine one another uniquely (since orthogonal projections are  determined by their kernels). It also shows that $P(B)=0$ if and only if $\sigma_h(B) =0$ for all $h \in \H$ or, equivalently, $P(B)=\id_\H$ if and only if $\sigma_h$ is supported in $B$ for all $h \in \H$. 
\end{rmk}

\subsection{Dynamical spectrum}

\begin{defn}\label{spectraldef}
Let $\alpha$ be a continuous action of a  second-countable locally compact abelian group $G$ on a second-countable locally compact  space $X$.  Fix an invariant measure $\m \in \MeasG(X)$.
\begin{enumerate}[(a)]
\item $U_\m$ denotes the corresponding strongly-continuous    representation  of $G$ on $L^2(X,\m)$ defined by $(U_\m(t)h)(x)=h(\alpha(-t,x))$ for $t \in G$, $x \in X$.
\item $\pi_\m : \Cz(\widehat G) \to \bdd(L^2(X,\m))$ denotes  the \textbf{associated $*$-representation}. Note that, if  $f=\widehat g$ for $g \in \Cc(G)$ and $\varphi  \in \Cz(X)$, then $\pi_\m(f)[\varphi]_\m = [g*\varphi]_\m$ where the convolution $g*\varphi \in \Cz(X)$ is defined by 
\[
(g*\varphi)(x) \coloneqq \int_G g(t) \varphi(\alpha(-t,x)) \, \dd t \qquad \forall x \in X \,.
\]
\item  $\sigma^\m_h$ denotes the \textbf{spectral measure} on $\widehat G$ associated to $h \in L^2(X,\m)$. 

Note that 
\[
\| \sigma^\m_h\| = \|h\|^2_\m
\]
so that $\sigma^\m_h \in \Meas(\widehat G)$ if $\|h\|_\m \leq 1$. 

Note also that, if $f = \widehat g$ where $g \in \Cc(G)$ and $\varphi \in \Cz(X)$, then 
\[
\int_{\widehat G}f(\chi) \, \dd \sigma^\m_\varphi(\chi) = \int_X (g*\varphi)(x)\overline{\varphi(x)} \, \dd \m(x) 
\,,
\]
where $(g*\varphi)\overline \varphi \in \Cz(X)$.
\item $\overline \pi_\m : \Borel(\widehat G) \to \bdd(L^2(X,\m))$ denotes the \textbf{canonical Borel extension} of $\pi_\m$ to the algebra of bounded Borel functions on $\widehat G$ determined by  
\[
\< \overline\pi(f)h,h \> = \int_{\widehat{G}} f(\chi) \, \dd \sigma^\m_h(\chi) \qquad \forall  f \in \Borel(\widehat G), h \in \H \,.
\]
\item $P_\m$ denotes the  associated \textbf{projection-valued measure} defined by  
\[
P_\m(B)\coloneqq \overline\pi_\m(\1_B)
\]
for Borel subsets $B \subset \widehat G$.
\item We say $\m$ is \textbf{spectrally supported} in a Borel set $B \subset \widehat G$ if $\sigma^\m_h$ is supported in $B$ for all $h \in L^2(X,\m)$. See Remark~\ref{supprmk}  for equivalent formulations.
\end{enumerate}
\end{defn}

Let us recall the following standard definition:
\begin{defn}
Recall that a nonzero element $h \in L^2(X,\m)$ is called an \textbf{eigenfunction} with \textbf{eigenvalue} $\chi \in \widehat{G}$ if it satisfies
\begin{align*}
U(t) h = \chi(t)h  && \text{ for all } t \in G.
\end{align*}
We denote by $\SSS$ the set of eigenvalues. Note that when $\m$ is ergodic, $\SSS$ is a group.

For each $\chi \in \SSS$ we denote by 
\[
E_\chi:= \{ f \in L^2(X, \m) : f \mbox{ is eigenfunction for } \chi \} \cup \{ 0\} \,,
\]
which is the \textbf{eigenspace} for $\chi$.
\end{defn}

We define 
\begin{align*}
\left(L^2(X, \m)\right)_{\mathsf{pp}}&:= \{ f \in L^2(X, \m) : \sigma_{f} \mbox{ is pure point } \} \\
\left(L^2(X, \m)\right)_{\mathsf{c}}&:= \{ f \in L^2(X, \m) : \sigma_{f} \mbox{ is continuous } \} \\
\left(L^2(X, \m)\right)_{\mathsf{ac}}&:= \{ f \in L^2(X, \m) : \sigma_{f} \mbox{ is absolutely continuous } \} \\
\left(L^2(X, \m)\right)_{\mathsf{sc}}&:= \{ f \in L^2(X, \m) : \sigma_{f} \mbox{ is singular continuous } \} \,.
\end{align*}

The following is well known, see \cite[Theorem 8.49]{ChevRay}.

\begin{fact*} We have 
\begin{align*}
L^2(X, \m)&= \left(L^2(X, \m)\right)_{\mathsf{pp}} \bigoplus \left(L^2(X, \m)\right)_{\mathsf{c}}\\
   \left(L^2(X, \m)\right)_{\mathsf{c}}&= \left(L^2(X, \m)\right)_{\mathsf{ac}}\bigoplus \left(L^2(X, \m)\right)_{\mathsf{sc}} \\
   \left(L^2(X, \m)\right)_{\mathsf{pp}}&= \bigoplus_{\chi \in \SSS} E_\chi \,.
\end{align*}
\end{fact*}

Since $X$ is second-countable, $C(X)$ is separable, and hence so is $L^2(X, \m)$. Therefore, $\SSS$ is countable. In particular, there is always a countable orthonormal basis of eigenfunctions for $L^2(X,m)_{\mathsf{pp}}$.

\medskip

Note that the constant function $1$ always belongs to $E_0$ and hence $E_0$ is non-trivial. In particular, $\left(L^2(X, \m)\right)_{\mathsf{pp}}$ is always non-trivial. With this is mind, we can introduce the following definition. Note   we have used additive notation here, so the trivial character is denoted $0$.

\begin{defn}\label{specdecomp} We say that $(X,G, \m)$ has 
\begin{itemize}
    \item[(a)] \textbf{pure point spectrum} if $\left(L^2(X, \m)\right)=\left(L^2(X, \m)\right)_{\mathsf{pp}}$;
    \item[(b)] \textbf{continuous spectrum} if $\left(L^2(X, \m)\right)= E_0 \bigoplus \left(L^2(X, \m)\right)_{\mathsf{c}}$;
        \item[(c)] \textbf{absolutely continuous spectrum} if $\left(L^2(X, \m)\right)= E_0 \bigoplus \left(L^2(X, \m)\right)_{\mathsf{ac}}$;
            \item[(d)] \textbf{singular continuous spectrum} if $\left(L^2(X, \m)\right)= E_0 \oplus \left(L^2(X, \m)\right)_{\mathsf{sc}}$.
\end{itemize}
\end{defn}

\begin{rmk}\text{ }
\begin{enumerate}[(a)]
    \item  $(X,G, \m)$ has pure point spectrum if and only if for all $h \in L^2(X, \m)$ the measure $\sigma_h$ is pure point, if and only if the subspace spanned by the eigenfunctions is dense in $L^2(X, \m)$. Equivalently, $\m$ has pure point spectrum if there is a countable (hence  Borel) set   $C \subset \widehat G$ such that $\m$ is spectrally supported in $C$.
    \item $(X,G, \m)$ has (absolutely/singular) continuous spectrum if and only if for all $h \in L^2(X,\m)$ the measure $\sigma_{h}-\sigma_{h}(\{0\}) \delta_0$ is (absolutely/singular) continuous.
    \item If $\m$ is ergodic, then $\1_X$ is a basis for $E_0$. This implies that 
    \[
    \sigma_{h}(\{ 0 \})= \| P_0(h) \|^2 = \left| \int_{X} h(t) \dd t \right|^2 \,, 
    \]
    where $P_0$ is the orthogonal projection on $E_0$. 

    It is easy to see that in this case  $(X,G, \m)$ has (absolutely/singular) continuous spectrum if and only if for all $h \in L^2(X,\m)$ with $\int_{X} f(t) \dd t=0$, the measure $\sigma_{h}$ is (absolutely/singular) continuous.
\end{enumerate}
\end{rmk}

\section{Main results}\label{sect:main}

In this section, we prove the results described in the introduction. As usual, $X$   is a second-countable locally compact  space,  $G$ is a second-countable locally compact abelian group, and $\alpha$ is a continuous action of $G$ on $X$.

\begin{lemma}\label{specint}
Suppose that $\m \in \MeasG(X)$,  $\rho \in \Prob(\MeasG(X))$ are such that 
\[
\int_X f(y) \, \dd \m(y) = \int_{\MeasG(X)} \int_X f(x) \, \dd \omega(x)  \, \dd \rho(\omega) \qquad \forall f \in \Cz(X) \,.
\]
Fix $\varphi \in \Cz(X)$. 

For $f \in \Borel(\widehat G)$, define an affine function $\Psi_f : \MeasG(X) \to \C$ by 
\[
\Psi_f (\omega) \coloneqq \int_{\widehat G} f(\chi) \, \dd \sigma^\omega_\varphi(\chi) \,.
\]

Then the following statements hold:
\begin{enumerate}[(a)]
\item For all $f \in \Cz(\widehat G)$, we have $\Psi_f \in C(\MeasG(X))$,
\item $\omega \mapsto \sigma^\omega_\varphi: \MeasG(X) \to \Meas(\widehat G)$  is continuous for the weak-star topologies, 
\item for all $f \in \Borel(\widehat G)$ we have $\Psi_f \in \Borel(\MeasG(X))$, and
\item 
\[
\int_{\widehat G} f(\chi) \, \dd \sigma^\m_\varphi(\chi)  = \int_{\MeasG(X)} \int_{\widehat G} f(\chi) \, \dd \sigma^\omega_\varphi(\chi)  \, \dd \rho(\omega) \qquad \forall f \in \Borel(\widehat G) \,.
\]
\end{enumerate}
\end{lemma}
\begin{proof}
Without loss of generality, $\|\varphi\|_\infty=1$.
For all $\omega$ we have 
\[
\left| \Psi_f (\omega) \right|   \leq \| f \|_\infty \, \sigma_{\varphi}^\omega(\widehat{G}) \leq \|f \|_\infty \,.
\]
For (a), it suffices to verify $\Psi_f$ is continuous for a uniformly dense subset of $\Cz(\widehat G)$. 

Let $f = \widehat g$ for some $g \in \Cc(G)$. In this case, putting $F=(g*\varphi)\overline \varphi \in \Cz(X)$ where 
\[
(g*\varphi)(x) = \int_G g(t) \varphi(\alpha(-t,x)) \, \dd t \,,
\] 
by Definition~\ref{spectraldef} we have
\[ 
\Psi_f(\omega) = \int_{\widehat G} f(\chi) \, \dd \sigma^\omega_\varphi(\chi)=  \< \pi_\omega(f)  \varphi, \varphi \>_\omega = \int_X F(x)  \, \dd \omega(x) \, . 
\]
Since $F \in \Cz(X)$, the expression on the right is continuous as a function of $\omega$ by definition of the weak-star topology, proving (a). 

Clearly (a)  implies (b). 

For  (c), suppose that $f \in \Borel(\widehat G)$. Since 
\[
\Meas(X) \ni \omega \mapsto \sigma^\omega_\varphi \in \Meas(\widehat G)
\]
is continuous  and   
\[
\Meas(\widehat G)  \ni \sigma \mapsto  \int_{\widehat G} f(\chi) \, \dd  \sigma( \chi)  \in \C
\]
is Borel (see Lemma~\ref{measurability}), we have that $\Psi_f$ is Borel measurable as desired. 

Towards  (d), first note that if $f = \widehat g$ for $g \in \Cc(G)$ and $F \coloneqq (g*\varphi)\overline \varphi$ as before, then
\begin{align*} 
\int_{\widehat G} f(\chi) \, \dd \sigma^\m_\varphi(\chi) &= \int_X F(x) \, \dd \m(x) = \int_{\MeasG(X)} \int_X F(y) \, \dd \omega(y) \, \dd \rho(\omega) \\
&= \int_{\MeasG(X)} \int_{\widehat G} f(\chi) \, \dd \sigma^\omega_\varphi(\chi) \, \dd \rho(\omega)
\end{align*}
so (d) holds in this case. 

Because $f\mapsto \Psi_f$ is contractive, a simple continuity argument gives that (a) holds as well for all  $f \in \Cz(G)$. Equivalently,
\[ 
\int_{\widehat G} f(\chi) \, \dd \sigma^\m_\varphi(\chi) =   \int_{\Meas(\widehat G)} \int_{\widehat G} f(\chi) \, \dd \sigma(\chi) \, \dd \rho^\sharp(\omega)
\]
for all $f \in \Cz(\widehat G)$ where $\rho^\sharp \in \Prob(\Meas(\widehat G))$ denotes the pushforward of $\rho$ by $\omega \mapsto \sigma^\omega_\varphi$.  By Proposition~\ref{C(X)toB(X)}, it follows that  
\[
\int_{\widehat{G}} f(\chi) \, \dd \sigma^\m_\varphi(\chi) =   \int_{\Meas(\widehat G)} \int_{\widehat G} f(\chi) \, \dd \sigma(\chi) \, \dd \rho^\sharp(\sigma)
\]
for all $f \in \Borel(\widehat G)$ which amounts to (d) for all $f \in \Borel(\widehat G)$.
\end{proof}

As a consequence, we get the following theorem, which is the main result of the paper and leads to interesting consequences for mean almost periodic measures. 

\begin{thm}\label{thm:almosteverywhere}
Let $G$ be  a second-countable locally compact  abelian group acting continuously on a second-countable locally compact  space $X$. Suppose  $\m \in \ProbG(X)$ and  $\rho \in \Prob(\Erg(X))$ satisfy 
\[
\int_X f(x) \, \dd \m(x) = \int_{\Erg(X)} \int_X f(y) \, \dd \omega(y) \, \dd \rho(\omega) \qquad \forall f \in \Cz(X) \,.
\]
If $B \subset \widehat G$ is a Borel set such that $P_\m(B) = \id_{L^2(X,\m)}$, then  $P_\omega(B) = \id_{L^2(X,\omega)}$  for $\rho$-a.e. $\omega \in \Erg(X)$. 
\end{thm}
\begin{proof}
Let $N = X-D$ so that the statement to be proven is $P_\omega(N)=0$ for $\rho$-a.e. $\omega \in \Erg(X)$. 

Fix a sequence  $\varphi_n \in \Cz(X)$, $n=1,2,\ldots$ with  dense linear span. Then, for every $\omega \in \Meas(X)$, the sequence (of classes) $\varphi_n \in L^2(X,\omega)$ has dense linear span so that
\[ 
\{ \omega \in \Erg(X) : P_\omega(D)=0 \} = \textstyle{\bigcap_{n=1}^\infty} \{ \omega \in \Erg(X) : P_\omega(D)\varphi_n =0\} \,. 
\]
Since co-null sets are closed under countable intersections, we need only fix an arbitrary $\varphi \in \Cz(X)$ and prove that $P_\omega(N) \varphi =0$ for $\rho$-a.e. $\omega \in \Erg(X)$.  Recalling $\|P_\omega(N) \varphi\|_\omega^2 = \sigma^\omega_\varphi(N)$ and using Lemma~\ref{specint} to write
\[ \sigma^\m_\varphi(N) = \int_{\Erg(X)} \sigma^\omega_\varphi(N) \, \dd \rho(\omega), \]
we obtain our desired conclusion.
\end{proof}

\medskip 

Next, let us recall the following definition from \cite{LSS2}

\begin{defn} Let $(X,G)$ be a topological dynamical system and let $\cA$ be a F{\o}lner sequence. An element $x \in X$ is called a \textbf{Besicovitch almost periodic point} with respect to the F{\o}lner sequence $\cA$ if for all $f \in C(X)$ the function 
\[
G \ni t \to f_x(x):= f(\alpha(t,x))
\]
belongs to $\bap_{\cA}(G)$. For properties of Besicovitch almost periodic points, see \cite{LSS2}.

The set of Besicovitch almost periodic points in $X$ is denoted by 
\[
\bapx \,.
\]  
\end{defn}

\begin{rmk} Let $\mu \in \cM^\infty(G)$ and let $\cA$ be a van Hove sequence. Then, $\mu \in \Bap_{\cA}(G)$ if and only if $\mu$ is a Besicovitch almost periodic point in $\XX(\mu)$ with respect to $\cA$ \cite[Prop.~7.1]{LSS2}.
    
\end{rmk}

The following is an immediate application of the main theorem. 

\begin{thm}\label{thm:ppimpliesrgpp} Let $(X,G,\m)$ be any measurable dynamical system. Assume that  $L^2(X,\m)$ has pure point spectrum.  Then there exists some $Y \subseteq \Erg(X)$ with the following properties:
\begin{itemize}
    \item[(a)] $\rho(Y)=1$.
    \item[(b)] For all $\omega \in Y$, $L^2(X,\omega)$ has pure point spectrum.
    \item[(c)] If $\cA$ is any F\o lner sequence along which the Birkhoff ergodic theorem holds, then 
for all $\omega \in Y$ the set $\bapx$ is $\omega$-measurable and  
\[
\omega(\bapx) =1 \,.
\]
In particular, 
\[
\bapx \neq \emptyset \,.
\]
    \item[(d)] $\chi \in \widehat{G}$ is an eigenvalue for $\m$ if and only if there exists some $\omega \in Y$ such that $\chi$ is an eigenvalue for $\omega$. Moreover, in this case, the set of $\omega \in \Erg(X)$ such that $\chi$ is an eigenvalue for $\omega$ has positive measure.
 \end{itemize}
\end{thm}
\begin{rmk} If $\bapx$ is Borel measurable, then 
\[
\m(\bapx) = \int_{\Erg(X)} \omega(\bapx) \dd \rho(\omega) =1
\]
since $\omega(\bapx) =1$ for $\rho$-almost all $\omega$.
\end{rmk}
\begin{proof}
Let $\SSS \subset \widehat G$ be the spectrum of $L^2(X,\m)$. Then, $\SSS$ is countable and every $\chi \in \SSS$ is an eigenvalue for $L^2(X,\m)$.

We have that $P_\SSS = \id_{L^2(X,\m)}$ and it follows from Theorem \ref{thm:almosteverywhere} that there exists some set $Y \subseteq \Erg(X)$ such that $\rho(Y)=1$ and $P_\SSS=\id_{L^2(X,\omega)}$ for all $\omega \in Y_0$. In particular, all $\omega \in Y$ have countable spectrum contained in $\SSS$. (a) and (b) and the implication $\Longleftarrow$ in (d) follow. 

(c) follows from \cite[Thm.~4.7]{LSS2}.

We need to prove the implication $\Longrightarrow$ in (d). Pick an arbitrary $\chi \in \SSS$. Since $C(X)$ is dense in $L^2(X,\m)$ and $P_\m({\{\chi\}}) \neq 0$, there exists some $\phi \in C(X)$ with $\| \phi \|_{2, \m}=1$ such that
\[
P_\m({\{ \chi \}}) \varphi \neq 0 \,.
\]
Then,
\begin{align*}
0 &\neq \| P_\m({\{ \chi \}})\varphi \|^2_2= \sigma^\m_{ P_\m(\{ \chi \}) \varphi }(\widehat{G}) = \sigma^\m_{\varphi}|_{\{ \chi \}}(\widehat{G})= \sigma^\m_{\varphi}(\{ \chi \}) \,.
\end{align*}
Applying Lemma~\ref{specint} (d) with $f= 1_{\chi}$, the characteristic function of the singleton $\chi$, we get 
\[
0 \neq \sigma^\m_{\varphi}(\{ \chi \})= \int_{\MeasG(X)} \sigma^\omega_{\varphi}(\{ \chi \}) \, \dd \rho(\omega) \,.
\]
This implies that the set of $\omega \in \Erg(X)$ such that 
\[
\sigma^\omega_{\varphi}(\{ \chi \}) \neq 0 \,,
\]
has positive measure in $\Erg(X)$. Since $\rho(Y)=1$, this set intersects $Y$. In particular, there exists some $\omega \in Y$ such that
\[
\sigma^\omega_{\varphi}(\{ \chi \}) \neq 0 \,.
\]
Then, $\chi$ is an eigenvalue for $L^2(X,\omega)$.
\end{proof}

\begin{rmk} Under the assumptions of Theorem~\ref{thm:ppimpliesrgpp}, the spectrum of $\m$ is just the union of spectrum of $\omega \in Y$.
\end{rmk}

Let us next discuss the eigenvalues of $\m$. The following result gives a simple characterization of the eigenvalues and the bounded eigenfunctions.

\begin{thm}
Let $G$ be  a second-countable locally compact  abelian group acting continuously on a second-countable locally compact  space $X$. Suppose  $\m \in \ProbG(X)$ and  $\rho \in \Prob(\Erg(X))$ satisfy 
\[
\int_X f(x) \, \dd \m(x) = \int_{\Erg(X)} \int_X f(y) \, \dd \omega(y) \, \dd \rho(\omega) \qquad \forall f \in \Cz(X) \,.
\]
Let $f \in B(X)$ and $\chi \in \widehat{G}$. Then $f$ is an eigenfunction with eigenvalue $\chi$ if and only if there exists some set $Y \subseteq \Erg(X)$ with $\rho(Y)=1$ such that $f$ is an eigenfunction with eigenvalue $\chi$ in $L^2(X, \omega)$ for all $\omega \in Y$.
\end{thm}
\begin{proof}
Pick some dense countable set $D \subseteq G$.

Now, for each $t \in D$ by Proposition~\ref{C(X)toB(X)} we have 
\begin{equation}\label{eq:eigen}
\int_{X} \left| U_tf(x) - \chi(t) f(x) \right|^2 \dd \m(x)= \int_{\Erg(X)} \int_X \left| U_tf(y) - \chi(t) f(y) \right|^2 \, \dd \omega(y) \, \dd \rho(\omega) \,.
\end{equation}

$\Longleftarrow$: By assumption, for all $\omega \in Y$ we have  
\[
\int_X \left| U_tf(y) - \chi(t) f(y) \right|^2 \dd \omega(y) =0 \,.
\]
This implies that 
\[
\int_{X} \left| U_tf(x) - \chi(t) f(x) \right|^2 \dd \m(x)=0
\]
and hence
\[
U_t f = \chi(t) f \qquad \mbox{ in } L^2(X, \m)
\]
for all $t \in D$. Since $D$ is dense in $G$ and $U_t$ and multiplication by $\chi(t)$ are continuous, we get that $f$ is an eigenfunction with eigenvalue $\chi$.

$\Longrightarrow$ is similar. For each $t \in D$ we have 
\[
\int_X \left| U_tf(y) - \chi(t) f(y) \right|^2 \dd \omega(y) =0 \,.
\]
Then, \eqref{eq:eigen} implies that there exists a set $Y_t \subseteq \Erg(X)$ with $\rho(Y_y)=1$ such 
that for all $\omega \in Y_t$ we have  
\[
\int_X \left| U_tf(y) - \chi(t) f(y) \right|^2 \dd \omega(y) =0 \,.
\]
Then, the set 
\[
Y:= \cap_{t \in D} Y_t
\]
has full measure in $\Erg(X)$ and for all $\omega \in Y$ and $t \in D$ we have 
\[
\int_X \left| U_tf(y) - \chi(t) f(y) \right|^2 \dd \omega(y) =0 \,.
\]
Exactly as above, this gives that for all $\omega \in Y$ $f$ is an eigenfunction with eigenvalue $\chi$.
\end{proof}

Next, let us list two immediate consequences of Theorem~\ref{thm:ppimpliesrgpp}.
\begin{cor}
Let $(X, G,\m)$ be any measurable dynamical system. If $L^2(X,\m)$ 
has singular spectrum, then $\rho$-a.e. $\omega \in \Erg(X)$ has singular spectrum. \qed
\end{cor}

\begin{cor}
Let $(X, G,\m)$ be any measurable dynamical system. If the spectrum of $L^2(X,\m)$ has Hausdorff dimension at most $\alpha \geq0$ then, for 
$\rho$-a.e. $\omega \in \Erg(X)$ the spectrum of $L^2(X,\omega)$ has Hausdorff dimension at most $\alpha >0$. \qed
\end{cor}

We complete the section by looking at an interesting example.

\begin{example}
Let $\alpha$ be the homeomorphism of the two-dimensional torus $\T^2=\R^2/\Z^2$ defined, in additive notation, by the matrix $\mat{1 & 1 \\ 0 & 1}$. That is,  $\alpha(x,y) = (x+y,y)$. Let $\m \in \Prob(\T^2)$ be the  normalized Haar measure. For fixed $y \in \T$, let $\omega_y \in \Prob(\T^2)$ be the normalized Haar measure of $\T$ copied on the slice  $\T_y$. On the slice $\T_y$, $\alpha$ acts as rotation by the angle $y$. Thus, we may think of the dynamical system $(\T^2,\alpha)$ an aggregate of all the circle rotations, both rational and irrational. The ergodic probability measures for $\alpha$ are of two types:
\begin{itemize}
\item For irrational angles $y \in \T_\text{irrat}$, the rotation $x\mapsto x+y : \T \to \T$ is minimal with unique ergodic measure $\dd x$. Copying this measure on the circular slice $\T \times \{y\}$ yields an ergodic measure for $\alpha$. 
\item The remaining ergodic measures for $\alpha$ are normalized counting measures along finite orbits.   In more detail, for rational angles $y \in \T_\text{rat}$, the rotation $x \mapsto x+y: \T \to \T$ is $n$-periodic where   $n \in \Z^+$ is the denominator of $y$ expressed in lowest terms. The orbits are the fibers of the $n$-fold covering map $x \mapsto nx : \T \to \T$, so it is natural to parametrize them by $\T$.  
\end{itemize}
In summary, we have a   natural identification:
\[ \Erg(\T^2,\alpha) = \T_\text{irrat} \cup (\T_\text{rat} \times \T). \]
The usual Haar measure $\m = \dd x \, \dd y$ on $\T^2$ is  $\alpha$-invariant. 

It easy to show that the mapping $y \mapsto \dd x \times \delta_{y} : \T   \to \Prob(\T^2)$ is a homeomorphism  onto its image, but take care that  $\omega_\theta \notin \Erg(X)$ for $\theta$ rational. On the other hand, the homeomorphic image of $[0,1] \setminus \Q$ is a $G_\delta$
-set contained in $\Erg(X)$. Let $\rho$ be the probability measure on $\Erg(X)$ given as the pushforward through $\theta \mapsto \omega_\theta$ of the restriction to $[0,1]\setminus \Q$ of the usual arc length measure on $[0,1]$. It is easily to see that this is the desired measure $\rho$.

We get a corresponding unitary operator: 
\begin{align*} 
U : L^2(\T^2) \to L^2(\T^2) && (Uh)(x,y)=h(x+y,y).
\end{align*}
On the usual orthonormal basis $\{w^mz^n: m,n \in \Z\}$ where $w(x,y)=e^{2\pi ix}$, $z(x,y)=e^{2 \pi i y}$, the action of $U$ is given by
\[ U w^mz^n =  w^m z^{m+n}.  \]
In other words, under Fourier isomorphism $L^2(\T^2)\cong \ell^2(\Z^2)$, the unitary $U$ is conjugate to:
\begin{align*} 
V : \ell^2(\Z^2) \to \ell^2(\Z^2) && (Vh)(m,n)=h(m+n,n).
\end{align*}
That is, the unitary induced by the permutation of $\Z^2$ defined by the matrix $\mat{1 & 0 \\ 1 & 1}$. Thus,  $V$ may be rather straightforwardly understood as the direct sum over $m \in \Z$ of the $m$th powers of the bilateral shift $S : \ell^2(\Z) \to \ell^2(\Z)$. 

For the purposes of calculating the spectral measures of the $\Z$-action, however, it is somewhat more convenient to only take the Fourier transform in the $x$ variable. This is because,  under the resulting isomorphism $L^2(\T^2) \cong L^2(\Z \times \T)$, the unitary $U$ is conjugate to a multiplication operator
\begin{align*}
W :L^2(\Z \times \T) \to L^2(\Z \times \T) && (Wh)(m,y) = e^{2\pi i m y} h(m,y).
\end{align*}
Referring to Example~\ref{multopex}, the associated $*$-representation $\pi$ of $C(\T)$ on $L^2(\Z \times \T)$ is  also  an action by multiplication operators: 
\begin{align*} (\pi(f) h)(m,y) = f(my) h(m,y) && f\in C(\T), h \in L^2(\Z \times \T)\,.
\end{align*}
Given a $h \in L^2(\Z \times \T)$, the associated spectral  measure $\sigma_h^\m$ on $\T$ is the pushforward of  $|h(m,y)|^2 dy$ by the map 
\begin{align*}
q : \Z \times \T \to \T && q(m,y)=my.
\end{align*}
Define $\eta \in L^1_+(\T)$ and $c \geq 0$ by:
\begin{align*}
\eta(x) \coloneqq \sum_{m \neq 0} \frac{1}{|m|} \sum_{q(m,y) = x} |h(m,y)|^2 && c \coloneqq \int_\T |h(0,y)|^2 \, \dd y.
\end{align*}
Then, we have
\[ \sigma_h^\m = \eta(x) \dd x + c \delta_0 \]
where $\dd x$ is the Haar measure of $\T$ and $\delta_0$ is the Dirac measure at the identity.
\end{example}

\section{Dynamical and diffraction spectrum of translation bounded measures}\label{sec:Bap}

\subsection{Dynamical spectrum}

Let us recall the following definition of \cite{BL}.

\begin{defn} By a \textbf{dynamical system of translation bounded measures or (TMDS)} on short, we mean a pair $(\XX,G,\m)$ where $\XX$ is a $G$-invariant vaguely compact subset of $\cM^\infty(G)$.
\end{defn}
Whenever when $\mu \in \cM^\infty(G)$ the hull $\XX(\mu)$ gives rise to a TMDS.

Let us start with two immediate consequences of Theorem~\ref{thm:almosteverywhere}.

\begin{thm} Let $(\XX,G)$ be a TMDS and let $\m$ be any $G$-invariant probability measure on $\XX$. 

If the dynamical spectrum of $(\XX, \m)$ is supported inside the Borel set $B \subseteq \widehat{G}$ then so is the spectrum of $(\XX, \omega)$ for $\rho$-a.s. all $\omega \in \Erg(\XX)$.\qed 
\end{thm}

\begin{thm}Let $(\XX,G)$ be a TMDS and let $\m$ be any $G$-invariant probability measure on $\XX$. 

If the dynamical spectrum of $(\XX, \m)$ is singular, the for $\rho$-a.s. all $\omega \in \Erg(\XX(\mu))$ the spectrum of $(\XX, \omega)$ is singular. \qed 
\end{thm}

Next, the following is an immediate consequence of Theorem~\ref{thm:ppimpliesrgpp} and \cite[Prop.~7.1]{LSS2}.

\begin{thm}\label{thm:bap} Let $(\XX,G)$ be a TMDS and let $\m$ be any $G$-invariant probability measure on $\XX$. 

Assume that $L^2(\XX, \m)$ has pure point spectrum. Let $\rho$ be the probability  measure on $\Erg(\XX)$  representing $\m$. Then there exists some $Y \subseteq \Erg(\XX$ with the following properties:
\begin{itemize}
    \item[(a)] $\rho(Y)=1$.
    \item[(b)] For all $\omega \in Y$, $L^2(\XX,\omega)$ has pure point spectrum.
    \item[(c)] If $\cA$ is any van Hove sequence along which the Birkhoff ergodic theorem holds, then 
for all $\omega \in Y$ the set $\XX \cap \Bap_{\cA}(G)$
\[
\omega(\XX \cap \Bap_{\cA}(G)) =1 \,.
\]
In particular, 
\[
\XX \cap \Bap_{\cA}(G) \neq \emptyset \,.
\]
    \item[(d)] $\chi \in \widehat{G}$ is an eigenvalue for $\m$ if and only if there exists some $\omega \in Y$ such that $\chi$ is an eigenvalue for $\omega$. Moreover, in this case, the set of $\omega \in \Erg(\XX)$ such that $\chi$ is an eigenvalue for $\omega$ has positive measure.
 \end{itemize}
\qed
\end{thm}

As a consequence, we get:
\begin{cor} Let $\mu$ be a translation bounded measure and let $\cA$ be a van Hove sequence. 

If $\mu \in \Map_{\cA}(G)$ then there exists some subsequence $\cA'$ of $\cA$ such that 
\[
\XX(\mu) \cap \Bap_{\cA'}(G) \neq \emptyset \,.
\]
\end{cor}
\begin{proof}
By Theorem~\ref{thm:map-gen}  a measure $\m$ with pure point spectrum on $\XX(\mu)$. Also, by  \cite[Theorem 1.2 and Pro.~1.4]{Lin} $\cA$ has a tempered subsequence $\cA'$, and that the Birkhoff ergodic theorem holds along $\cA'$.

The claim follows from Theorem~\ref{thm:bap}.
\end{proof}

\bigskip

We can now show that each mean almost periodic measure is Besicovitch almost periodic along a translated subsequence of $\cA$. First, we need the following preliminary result.

\begin{propn} Let $G$ be a second-countable locally compact abelian group. Fix  a translation bounded measure $\mu \in \cM^\infty(G)$ and a van Hove sequence $\cA = \{A_n\}$. If the  hull $\XX(\mu)$ contains a measure $\nu$ that is Besicovitch almost periodic with respect to $\cA$, then there exist group elements $t_n \in G$ such that $\mu$ is Besicovitch almost periodic with respect to $\cB = \{B_n\}$, where $B_n = t_n+A_n$.
\end{propn}
\begin{proof} Since $G$ is second-countable, there exists a countable dense  set $D=\{ \varphi_m \}_m$  in $\Cc(G)$ as in Lemma~\ref{lem-D-exists}.

\smallskip
Now, since $\nu \in \Bap_{\cA}(G)$, for each $m,l$ there exists some $f_{m,l} \in SAP(G)$ such that
\[
\| \nu*\varphi_m - f_{m,l} \|_{\mathsf{b},1,\cA} < \frac{1}{l} \,.
\]
Next, pick some $s_n$ such that 
\[
T_{s_n} \mu \to \nu \,.
\]

\smallskip

Next, for each $m,l$ the orbit closure
\[
C_{m,l}= \overline{\{ T_t f_{m,l}: t \in G\}} \subseteq (\Cu(G), \| \,\|_\infty)
\]
is compact. It follows that 
\[
C:=\prod_{m,l} C_{m,l}
\]
is a compact metrizable space. Therefore, the sequence $(T_{s_n} f_{m,l})_{m,l} \in C$ has subsequence convergent to some $(g_{m,l})_{m,l} \in C$. Therefore, there exists some subsequence $s_{k_n}$ of $s_n$ and  $g_{m,l} \in SAP(G)$ such that
\begin{equation}\label{eq1}
    \lim_n \| T_{s_{k_n}} f_{m,l} -g_{m,l} \|_\infty = 0 \qquad \forall m,l \,.
\end{equation}
In particular, for each $m,l$ and each $t \in G$ we have 
\begin{equation}\label{eq2}
\lim_n g_{m,l}(s_{k_n}+t)= f_{m,l}(t)  \,.
\end{equation}

\medskip

Next, for each fixed $m,l,j$ define 
\[
h_{n}(t):=\frac{1}{|A_j|} 1_{A_j}(t) \left|g_{m,l}(t+s_{k_n})-\mu*\varphi_m(t+s_{k_n})\right| \,.
\]
Then, $h_n$ is bounded by $\frac{\|g_{m,l}\|_\infty + \| \mu*\varphi_m \|_\infty}{|A_j|} 1_{A_j} \in L^1(G)$. Moreover, since $T_{s_n} \mu \to \nu$ in the vague topology, we have 
\[
\mu*\varphi_m(t+s_{k_n})\stackrel{\mbox{pointwise}}{\longrightarrow}\nu*\varphi_m(t) \,.
\]
Also, by \eqref{eq2}
\[
g_{m,l}(t+s_{k_n}) \stackrel{\mbox{pointwise}}{\longrightarrow} f_{m,l}(t)\,.
\]

Therefore, by the dominated convergence theorem, we have for each fixed $m,l,j$
\[
\lim_n \frac{1}{|A_j|} \int_{A_j} \left|g_{m,l}(t+s_{k_n})-\mu*\varphi_m(t+s_{k_n})\right|\, \dd t = \frac{1}{|A_j|} \int_{A_j} \left|f_{m,l}(t)-\nu*\varphi_m(t)\right| \, \dd t \,.
\]

Now, pick some $j>1$. Then, there exists some $N_j$ such that, for all $1 \leq m,l \leq j$ and all $n \geq N_j$ we have 
\[
\frac{1}{|A_j|} \int_{A_j} \left|g_{m,l}(t+s{k_n})-\mu*\varphi_m(t+s_{k_n})\right|\, \dd t \leq \frac{1}{|A_j|} \int_{A_j} \left|f_{m,l}(t)-\nu*\varphi_m(t)\right| \, \dd t +\frac{1}{j} \,.
\]
We can also assume without loss of generality that $N_j$ is strictly increasing.

Define
\[
t_j:= -s_{k_{N_{j}}} \,.
\]
Then, for all $1 \leq m,l \leq j$ we have 
\begin{align*}
\frac{1}{|A_j|} \int_{t_j+A_j} \left|g_{m,l}(t)-\mu*\varphi_m(t)\right|\, \dd t &=
\frac{1}{|A_j|} \int_{A_j} \left|g_{m,l}(t+s{k_{N_j}})-\mu*\varphi_m(t+s_{k_{N_j}})\right|\, \dd t \\
&\leq \frac{1}{|A_j|} \int_{A_j} \left|f_{m,l}(t)-\nu*\varphi_m(t)\right| \, \dd t +\frac{1}{j}  \,.
\end{align*}

Letting $j \to \infty$ we get 
\[
\| g_{m,l}- \mu*\varphi_m \|_{\mathsf{b},1,\cB} \leq \| \nu*\varphi_m - f_{m,l} \|_{\mathsf{b},1,\cA} < \frac{1}{l} \,.
\]
This shows that $\mu*\varphi_m \in \bap_{\cB}(G)$ for all $\varphi_m$. 

Therefore, $\mu \in \Bap_{\cB}(G)$ by Lemma~\ref{lemma-D-used}.

\end{proof}

Combining all results in this section, we get:

\begin{thm}\label{thm-map-is-bap} Let $G$ be a second-countable locally compact abelian group. Suppose that a translation bounded measure $\mu \in \cM^\infty(G)$ is mean almost periodic with respect to a van Hove sequence $\cA=\{ A_n \}$. Then, there exists a subsequence $A_{k_n}$ of $A_n$ and group elements $t_n \in G$ such  that $\mu$ is Besicovitch almost periodic with respect to $\cB=\{B_n\}$ where  $B_n=t_n+A_{k_n}$. \qed
\end{thm}

The following are trivial consequences of  Theorem~\ref{thm-map-is-bap} and  \cite[Proposition 3.33]{LSS}.

\begin{thm} 

Let $\cA=\{ A_n \}$ be a van Hove sequence and let $f  \in \Cu(G) \cap \map_{\cA}(G)$.

Then, there exists a subsequence $A_{k_n}$ of $A_n$ and some $t_n \in G$ such that, setting $B_n=t_n+A_{k_n}$ and $\cB=\{B_n\}$ we have
\[
f \in \bap_{\cB}(G) \,.
\]\qed
\end{thm}

\begin{cor} Let $f \in \Cu(G)$ and let $\cA$ be a van Hove sequence. If $f \in \map_{\cA}(G)$ then, there exists some $g \in \XX(f) \cap \bap_{\cA}(G)$. \qed
\end{cor}

\subsection{Diffraction spectrum}

Consider a TMDS $(\XX,G)$. Each $\varphi \in \Cc(G)$ defines a continuous function $f_\varphi : \XX \to \C $ via
\[
f_\varphi(\omega):= \omega*\varphi(0) \, .
\]
Let us recall the following result: 

\begin{thm}\cite[Prop.~6]{BL}\label{thm5} Let $(\XX, G)$ be a (TMDS) and $\m$ a $G$-invariant probability measure on $\XX$. Then, there exists a positive definite measure $\gamma_{\m}$ on $G$ satisfying  
\[
\< f_\varphi, U_t f_\varphi \> = \gamma_{\m}*\varphi*\widetilde{\varphi}(t) 
\]
for all $\varphi \in \Cc(G)$ and all $t \in G$. \qed
\end{thm}

Since $\gamma_{\m}$ is positive definite, by \cite{ARMA1,BF} there exists a positive measure\footnote{The measure $\widehat{\gamma_{\m}}$ is called the \textbf{Fourier transform of $\gamma$} \cite{ARMA1}.} $\widehat{\gamma_{\m}}$ on $\widehat{G}$ such that for all $\varphi \in \Cc(G)$ we have $|\widehat{\varphi}|^2 \in L^1(\widehat{\gamma_{\m}})$ and $\gamma_{\m}*\varphi*\widetilde{\varphi}$ is the continuous positive definite function associated via the Bochner theorem to the finite positive measure $|\widehat{\varphi}|^2 \widehat{\gamma_{\m}}$.

\begin{defn} The measure $\gamma_{\m}$ from Theorem~\ref{thm5} is called the \textbf{autocorrelation measure} of  $(\XX, G, \m)$ while the measure $\widehat{\gamma_{\m}}$ is called the \textbf{diffraction measure} of $(\XX, G, \m)$.
\end{defn}

\begin{rmk}
\begin{itemize}
    \item[(a)] Given some $\mu \in \cM^\infty(G)$ and a van Hove sequence $\cA$, an autocorrelation $\gamma$ of $\mu$ is usually defined as a vague cluster point of 
\[
\frac{1}{|A_n|} \mu|_{A_n}*\widetilde{\mu|_{A_n}} \,.
\]
When $G$ is second-countable, there is always a subsequence $A_{k_n}$ of $A_n$ such that the vague limit 
\[
\gamma=\lim_n \frac{1}{|A_{k_n}|} \mu|_{A_{k_n}}*\widetilde{\mu|_{A_{k_n}}} \,.
\]
By eventually replacing $A_n$ by $A_{k_n}$ we can usually assume without loss of generality that the autocorrelation $\gamma$ exists along $\cA$.
\item[(b)] Assume that the autocorrelation of $\mu \in \cM^\infty(G)$ exists along $\cA$. Let $\m$ be  measure on $\XX(\mu)$ given by Theorem~\ref{thm:map-gen}.

Then $\gamma$ is the autocorrelation of $(\XX(\mu),G, \m)$ \cite{LSS,LS}.
\end{itemize}
\end{rmk}

Let us now recall the so called Dworkin argument, see for example \cite{BL,Gou,LMS}.

\begin{propn} Let $(\XX, G)$ be a (TMDS) and $\m$ a $G$-invariant probability measure on $\XX$. Then, for all $\varphi \in \Cc(G)$ we have 
\[
\sigma_{_\varphi}^\m = \left| \widehat{\varphi} \right|^2 \widehat{\gamma_{\m}} \,.
\]\qed
\end{propn}

\begin{rmk} The diffraction spectrum can be seen as the dynamical spectrum of the sub-representation $U$ of $G$ on the $G$-invariant Hilbert space (see \cite[Lemma~8 and Theorem~6]{BL})
\[
\mathcal{H}= \overline{ \{f_\varphi : \varphi \in \Cc(G) \}} \subseteq L^2(\XX, \m) \,.
\]
In general the diffraction spectrum is smaller than the dynamical spectrum, and it can be strictly smaller.

On another hand, pure point diffraction and dynamical spectrum are equivalent and \cite[Thm.~7, Thm.~9]{BL}
$\SSS$ is the group generated by 
\[
\{ \chi \in \widehat{G} : \widehat{\gamma_\m}(\{ \chi \}) >0 \,. 
\]
\end{rmk}

\medskip

Now, we get the following consequence of Lemma~\ref{specint}, which can  be seen as the diffraction version of Thm.~\ref{thm:almosteverywhere}.

\begin{thm} Let $(\XX,G)$ be a TMDS and let $\m$ be any $G$-invariant probability measure on $\XX$. 

Let $\rho$ be the probability  measure on $\Erg(\XX)$  representing $\m$. Then, 
\begin{itemize}
    \item[(a)] For all $\varphi \in \Cc(G)$ and Borel sets $B \subseteq \widehat{G}$ we have 
\[
\int_{B} \left| \widehat{\varphi}(\chi) \right|^2  \dd \widehat{\gamma_\m}(\chi) = \int_{\Erg(\XX)} \int_{B} \left| \widehat{\varphi}(\chi) \right|^2 \dd \widehat{\gamma_\omega}(\chi) \dd \rho(\omega) \,. 
\]
\item[(b)] For all $h \in B(\widehat{G})$ with $\supp(h)$ compact we have 
\[
\int_{\widehat{G}} h(\chi)   \dd \widehat{\gamma_\m}(\chi) = \int_{\Erg(\XX)} \int_{\widehat{G}} h(\chi) \dd \widehat{\gamma_\omega}(\chi) \dd \rho(\omega) \,. 
\]
\end{itemize}
\end{thm}
\begin{proof}
(a) follows from Lemma~\ref{specint}(d) applied to $f=\1_B$ and $f_\varphi$.

For (b), pick an arbitrary $h \in B(\widehat{G})$ of compact support. 
Pick some compact set $K$ such that $\supp(h) \subseteq K^\circ$.

By \cite{BF,MoSt} there exists some $\varphi \in \Cc(G)$ such that $\widehat{\varphi} \geq 1_K$. We can then pick some $f \in B(\widehat{G})$ such that 
\[
f \cdot \left| \widehat{\varphi} \right|^2 = h \,.
\]
The claim follows now from Lemma~\ref{specint}(d) applied to $f$ and $f_\varphi$.

\end{proof}

Now, we list some of the consequences.

\begin{cor} Let $(\XX,G)$ be a TMDS and let $\m$ be any $G$-invariant probability measure on $\XX$. 

Let $\rho$ be the probability  measure on $\Erg(\XX)$  representing $\m$. Let $B \subseteq \widehat{G}$ be a Borel set, and $\chi \in \widehat{G}$. Then,
\begin{itemize}
\item[(a)] If $B$ is pre-compact then 
\[
\widehat{\gamma_\m}(B) = \int_{\Erg(\XX)} \widehat{\gamma_\omega}(B) \dd \rho(\omega) \,. 
\]
\item[(b)] 
\[
\widehat{\gamma_\m}(\{\chi\}) = \int_{\Erg(\XX)} \widehat{\gamma_\omega}(\{\chi\}) \dd \rho(\omega) \,. 
\]
    \item[(c)] If $\supp(\widehat{\gamma_\m}) \subseteq B$ then $\supp(\widehat{\gamma_\omega}) \subseteq B$ for $\rho$-a.s. all $\omega \in \Erg(\XX)$.
    \item[(d)] If $\widehat{\gamma_\m}(\{ \chi\})=0$ then  $\widehat{\gamma_\omega}(\{ \chi\})=0$ for $\rho$-a.s. all $\omega \in \Erg(\XX)$.
    \item[(e)] If $\widehat{\gamma_\m}(B)>0$ then $\widehat{\gamma_\omega}(B)>0$ for $\omega$ in a subset of $\Erg(\XX)$ of positive $\rho$-measure.    
    \item[(f)] If $\widehat{\gamma_\m}(\{ \chi\})>0$ then  $\widehat{\gamma_\omega}(\{ \chi\})>0$ for $\omega$ in a subset of $\Erg(\XX)$ of positive $\rho$-measure.  
    \item[(g)] If $\widehat{\gamma_\m}$ is pure point then $\widehat{\gamma_\omega}$ is pure point for $\rho$-a.s. all $\omega \in \Erg(\XX)$.
    \item[(h)] If $\widehat{\gamma_\m}$ is singular then $\widehat{\gamma_\omega}$ is singular for $\rho$-a.s. all $\omega \in \Erg(\XX)$.
\end{itemize}\qed 
\end{cor}

\section*{Acknowledgements} C.R. was supported by the NSERC Discovery grant 2019-05430, and N.S. was supported by the NSERC Discovery grant 2024-04853.

\bibliographystyle{habbrv}
\bibliography{mainpaper}

\end{document}